\documentclass{amsart}

\usepackage{amsfonts}
\usepackage{amsmath}
\usepackage{amsrefs}
\usepackage{amssymb}
\usepackage{amsthm}
\usepackage[english]{babel}
\usepackage{color}
\usepackage{epsfig}
\usepackage{enumerate}
\usepackage{enumitem}
\usepackage{graphicx}
\usepackage{hyperref}
\usepackage[utf8]{inputenc}
\usepackage{mathrsfs}
\usepackage{multicol}
\usepackage{nicefrac}
\usepackage{psfrag}
\usepackage{tikz-cd}
\usepackage{ushort}

\hypersetup{
    colorlinks,
    citecolor=blue, 
    filecolor=blue,
    linkcolor=blue,
    urlcolor=blue
}

\newtheoremstyle{paragraph}{}{}{\normalfont}{}{}{}{\parindent}{}
\theoremstyle{plain}
\newtheorem{thm}{Theorem}[section]

\newtheorem{lem}[thm]{Lemma}
\newtheorem{prop}[thm]{Proposition}

\theoremstyle{definition}
\newtheorem{df}[thm]{Definition}
\theoremstyle{remark}

\newtheorem{conj}[thm]{Conjecture}

\newtheorem{ex}[thm]{Example}
\theoremstyle{paragraph}

\newcommand{\Z}{\mathbb Z}
\newcommand{\N}{\mathbb N}

\DeclareMathOperator{\twedge}{\textstyle\bigwedge}

\DeclareMathOperator{\toplus}{\textstyle\bigoplus}

\DeclareMathOperator{\tcap}{\textstyle\bigcap}

\DeclareMathOperator{\tsum}{\textstyle\sum}
\DeclareMathOperator*{\dsum}{\displaystyle\tsum}

\makeatletter
\providecommand{\bigsqcap}{%
  \mathop{%
    \mathpalette\@updown\bigsqcup
  }%
}
\newcommand*{\@updown}[2]{%
  \rotatebox[origin=c]{180}{$\m@th#1#2$}%
}
\makeatother

\makeatletter
\DeclareFontFamily{U}  {MnSymbolF}{}
\DeclareSymbolFont{symbolsMN}{U}{MnSymbolF}{m}{n}
\SetSymbolFont{symbolsMN}{bold}{U}{MnSymbolF}{b}{n}
\DeclareFontShape{U}{MnSymbolF}{m}{n}{
    <-6>  MnSymbolF5
   <6-7>  MnSymbolF6
   <7-8>  MnSymbolF7
   <8-9>  MnSymbolF8
   <9-10> MnSymbolF9
  <10-12> MnSymbolF10
  <12->   MnSymbolF12}{}
\DeclareFontShape{U}{MnSymbolF}{b}{n}{
    <-6>  MnSymbolF-Bold5
   <6-7>  MnSymbolF-Bold6
   <7-8>  MnSymbolF-Bold7
   <8-9>  MnSymbolF-Bold8
   <9-10> MnSymbolF-Bold9
  <10-12> MnSymbolF-Bold10
  <12->   MnSymbolF-Bold12}{}
\DeclareMathSymbol{\tbigtimes}{\mathop}{symbolsMN}{2}
\newcommand*{\bigtimes}{%
  \DOTSB
  \tbigtimes
  \slimits@ 
}
\makeatother

\makeatletter
\newsavebox\myboxA
\newsavebox\myboxB
\newlength\mylenA

\newcommand*\xbar[2][0.75]{%
    \sbox{\myboxA}{$\m@th#2$}%
    \setbox\myboxB\null
    \ht\myboxB=\ht\myboxA%
    \dp\myboxB=\dp\myboxA%
    \wd\myboxB=#1\wd\myboxA
    \sbox\myboxB{$\m@th\overline{\copy\myboxB}$}
    \setlength\mylenA{\the\wd\myboxA}
    \addtolength\mylenA{-\the\wd\myboxB}%
    \ifdim\wd\myboxB<\wd\myboxA%
       \rlap{\hskip 0.5\mylenA\usebox\myboxB}{\usebox\myboxA}%
    \else
        \hskip -0.5\mylenA\rlap{\usebox\myboxA}{\hskip 0.5\mylenA\usebox\myboxB}%
    \fi}
\makeatother

\title{Algebraic expansivity on abelian groups}
\author{Mauricio Achigar}
\date{\today}
\thanks{Partially supported by SNI-ANII (Uruguay) grant SNI\_2023\_1\_1013336.}

\begin{document}

\begin{abstract}
 Building on the author’s earlier work on topological and abstract expansivity, this paper introduces and explores the notion of \emph{algebraic expansivity} for endomorphisms of abelian groups. We analyze the fundamental properties of this algebraic analogue, establish its relationship with Weiss's algebraic entropy, and prove that positively expansive epimorphisms are necessarily restricted to finite systems. Finally, we demonstrate a robust connection with topological dynamics via Pontryagin duality: algebraic expansivity on torsion abelian groups is shown to be exactly the dual property of topological expansivity on totally disconnected compact groups.
\end{abstract}

\maketitle


\section{Introduction}

The study of expansive dynamical systems dates back to the work of Utz \cite{Utz50} in 1950, where \emph{expansive homeomorphisms}\footnote{This term was coined in \cite{GoH55}*{\S31.10}} were first introduced under the name \emph{unstable homeomorphisms}. A homeomorphism $f\colon X\to X$ of a metric space $(X,d)$ is expansive if there exists $c>0$ (called an \emph{expansivity constant}) such that
\begin{equation}\label{def:homeoexp_1}
 \text{if $x,y\in X$ and $d(f^kx,f^ky)\leq c$ for all $k\in\Z$ then $x=y$.}
\end{equation}

Although the theory has been extensively developed in different directions and contexts—for example, for smooth manifolds, topological groups, continuous flows, and variations such as $cw$-expansivity or entropy-expansivity, just to mention a few—in this introduction, we are mainly interested in tracing the evolution of the above definition within the context of compact topological spaces and how it is carried over to the algebraic setting of abelian groups.

In his 1960 paper \cite{Bry60}, Bryant gave a definition of expansivity for a homeomorphism $f$ of a (compact) uniform space $(X,\mathscr{U})$. The homeomorphism $f$ is expansive if there exists $\mathcal U\in\mathscr U$ (called an \emph{expansive index}) such that
\begin{equation}\label{def:homeoexp_2}
 \text{if $x,y\in X$ and $(f^kx,f^ky)\in\overline{\mathcal U}$ for all $k\in\Z$ then $x=y$,}
\end{equation}
where $\overline{\mathcal U}$ denotes the closure of $\mathcal U$ in $X\times X$. Note that (\ref{def:homeoexp_2}) is equivalent to 
\begin{equation*}\label{def:homeoexp_2'}
 \textstyle\bigcap_{k\in\Z}(f\times f)^k(\overline{\mathcal U})=\Delta_X, \tag{2'}
\end{equation*}
where $\Delta_X$ denotes the diagonal of $X\times X$.
In 1987, Fried \cite{Fri87} extended Bryant's definition to arbitrary (compact) spaces by dropping the requirement for an underlying uniform structure and requiring only that $\mathcal U$ be a neighborhood of $\Delta_X$.

Another extension of Bryant's definition to general (compact) spaces was formulated by Keynes and Robertson \cite{KeR69} in 1969. A homeomorphism $f\colon X\to X$ is expansive if there exists an open cover $\alpha$ (called a \emph{generator}) such that
\begin{equation}\label{def:homeoexp_3}
 \text{if $x,y\in X$ and $\{f^kx,f^ky\}\prec\overline{\alpha}$ for all $k\in\Z$ then $x=y$,}
\end{equation}
where $\overline{\alpha}=\{\overline U:U\in\alpha\}$, and $\{f^kx,f^ky\}\prec\overline{\alpha}$ means that $\{f^kx,f^ky\}\subseteq \overline U$ for some $U\in\alpha$. It can be shown that in this definition, (\ref{def:homeoexp_3}) can be replaced by 
\begin{equation*}\label{def:homeoexp_3'}
 \textstyle\twedge_{k\in\Z}f^k(\overline\alpha)=\delta_X, \tag{3'}
\end{equation*}
where $\delta_X$ is the cover of $X$ by its singletons\footnote{To be precise $\delta_X=\bigl\{\{x\}:x\in X\bigr\}\cup\{\varnothing\}$.} and $\wedge$ represents the operation on families of subsets of $X$ given by $\{A_i\}\wedge\{B_j\}=\{A_i\cap B_j\}$.

While the various definitions considered above are known to be equivalent for compact spaces, the formulation by Keynes and Robertson using open covers highlighted a connection that deeply linked expansivity and entropy. They discovered that, heuristically speaking, expansivity is equivalent to possessing a generator for the entropy---a generator whose orbit is sufficiently representative of the dynamics, in such a way that its  entropy coincides with the total entropy of the system.

In their seminal work where \emph{topological entropy} was first introduced via open covers, Adler, Konheim, and McAndrew \cite{AKM65} remarked:
\vspace{1.5mm}
\begin{center}
 \begin{minipage}{.92\textwidth}
  \emph{\quad The notion of entropy has an abstract formulation which we have not dealt with here. It can be tailored to fit mappings on other mathematical structures.}
 \end{minipage}
\end{center}
\vspace{1.5mm}
They explicitly suggested that, for the case of abelian groups, one could replace open covers with subgroups, among other modifications, thereby foreshadowing a theory analogous to topological entropy within an algebraic framework.

The challenge posed by Adler, Konheim, and McAndrew was addressed by Weiss \cite{Wei74} in 1974, who conducted the first systematic study of what he termed \emph{algebraic entropy} in the context of abelian groups. Notably, Weiss proved that the algebraic entropy of an automorphism of a torsion group coincides with the topological entropy of its Pontryagin dual. Furthermore, Lindenstrauss and Weiss \cite{LiW00} introduced in 2000 the \emph{mean dimension}---a topological invariant for homeomorphisms of compact spaces suggested by Gromov---whose theory was developed through techniques closely related to those of Adler et al.

More recently, Dikranjan and Giordano Bruno \cites{DGB13, DGB19} provided an abstract construction of entropy in the context of general normed semigroups, called \emph{semigroup entropy}. This abstract generalization has proven to be remarkably fruitful, encompassing, as special cases, topological entropy, Weiss' algebraic entropy, and mean dimension, among many other concrete examples.

Given the aforementioned generalizations of topological entropy and the insights provided by Keynes and Robertson, it is natural to pursue analogous generalizations for the notion of expansivity, while preserving the fundamental connection between these two concepts. To achieve this, one must first establish a definition of topological expansivity that avoids any reference to individual points or the closure of subsets, relying instead solely on more primitive structures, such as neighborhoods of the diagonal or open covers—much like the definition of topological entropy itself.

A first attempt might involve considering weakened versions of (2) and (3) by omitting the closures in those formulas, closures which correspond to the $\leq$ in (1). However, while a weakened version of Utz's definition---using $<$ instead of $\leq$ in (1)---is clearly equivalent to the original by suitably adjusting the constant $c$, the weakened versions of the definitions of Bryant, Fried, or Keynes-Robertson—using (2) or (3) without closures—also prove to be equivalent to the original ones (for compact Hausdorff spaces) by appropriately adjusting $\mathcal{U}$ or $\alpha$; nevertheless, this approach still fails to eliminate the inherent dependence on the points of the space.

A solution to the problem of finding a suitable definition was provided in the author's PhD thesis \cite{Ach19} and the paper \cite{AAM16} derived from its first chapter, where a new definition of expansivity on compact spaces was introduced under the name \emph{refinement expansivity}. This definition is based on generalizing the property of \emph{uniform expansivity} discovered by Bryant \cite{Bry60}, which is equivalent to the classical notion of expansivity: if $f\colon X\to X$ is an expansive homeomorphism as in Utz's definition, then for all $\varepsilon>0$ there exists $N\in\N$ such that
\begin{equation}\label{def:homeoexpunif}
 \text{if $x,y\in X$ and $d(f^kx,f^ky)\leq c$ for all $|k|\leq N$ then $d(x,y)\leq\varepsilon$,} 
\end{equation}
where $c>0$ is an expansivity constant.

The above condition, when expressed in terms of open covers\footnote{With an open cover $\alpha$ instead of $c$, and an open cover $\beta$ instead of $\varepsilon$.} in the same spirit that Keynes and Robertson formulated (3) or (3') from (1)---but omitting closures---translates into the following requirement on a homeomorphism $f$ acting on a compact space $X$: there exists an open cover $\alpha$ of $X$ with the property that for all open covers $\beta$ of $X$ there exists $N\in\mathbb{N}$ such that
\begin{equation}\label{def:homeorefexp}
 \text{if $x,y\in X$ and $\{f^kx,f^ky\}\prec\alpha$ for all $|k|\leq N$ then $\{x,y\}\prec\beta$.} 
\end{equation}
The definition of refinement expansivity is then obtained by rewriting (5) as
\begin{equation}\label{def:homeorefexp'}
 \twedge_{|k|\leq N}f^k(\alpha)\prec\beta, \tag{5'}
\end{equation}
where the notation $\{A_i\}\prec\{B_j\}$ means that each $A_i$ is a subset of some $B_j$.

Moreover, it is worth noting that the property defining a refinement expansive homeomorphism first appeared in the literature in the work of Keynes and Robertson; it is precisely the conclusion of Lemma 2.5 in \cite{KeR69}, where this property is shown to be satisfied by every generator. However, a simple inspection of their proofs reveals that the full strength of being a generator is not required to ensure that the entropy of a cover equals the total topological entropy of the system. Indeed, this refinement property is the only condition used to prove that this equality holds. This observation guarantees that refinement expansivity relates to entropy in the same way as the classical notion of expansivity.

The concept of refinement expansivity has shown itself to be particularly fertile, not only because it extends the classical notion—allowing for the study of new expansive dynamics such as the so-called non-Hausdorff shifts \cite{AAM16}—but also because it provided a foundation for generalizing the concept. This led to a definition of \emph{abstract expansivity}, akin to the framework of Dikranjan and Giordano Bruno \cites{DGB13, DGB19}, which is developed in a manuscript currently under review \cite{Ach23}. In that work, the concept of abstract expansivity is introduced alongside an abstract version of entropy.

By specializing this abstract framework to the context of abelian groups as indicated in \cite{Ach23}*{\S4.5.3}, we arrive at the \emph{algebraic expansivity} discussed in the present paper. This can also be achieved more directly by departing from refinement expansivity and following the approach used for Weiss's algebraic entropy by replacing open covers with subgroups, precisely as suggested by the foresight of Adler, Konheim, and McAndrew \cite{AKM65}. Perhaps unsurprisingly, it turns out that the resulting notion of algebraic expansivity for torsion abelian groups is the Pontryagin dual of the topological expansivity of the dual map.

\medskip

The organization of this paper is as follows. In Section~2, we introduce the definition of algebraic expansivity and present some of its basic properties. Section~3 is devoted to the relation between algebraic expansivity and Weiss's algebraic entropy. In Section~4, we examine positively expansive epimorphisms, showing that such systems must be finite. Finally, Section~5 deals with Pontryagin duality, relating algebraic expansivity to the topological expansivity on the dual group.

\section{Algebraic expansivity}

Throughout this work, all groups are assumed to be abelian unless otherwise stated.

\begin{df}\label{df:expansive}
 Let $\varphi\colon G\to G$ be an endomorphism of an abelian group $G$.
 
 We say that $\varphi$ is \emph{positively expansive} if there exists a finite subgroup $S\leq G$, called a \emph{positive generator} for $\varphi$, with the following property: for every finite subgroup $F\leq G$, there exists $n\in\N$ such that $F\subseteq\sum_{k=0}^n\varphi^kS$. In that case, we also say that $\varphi$ is \emph{positively $S$-expansive}.
 
 On the other hand, when $\varphi$ is an automorphism, if there is a finite subgroup $S\leq G$ such that for any given finite subgroup $F\leq G$ we have $F\subseteq\sum_{|k|\leq n}\varphi^kS$ for some $n\in\N$, we then say that $\varphi$ is \emph{expansive} (or \emph{$S$-expansive}) and $S$ is called a \emph{generator} for $\varphi$.
\end{df}

Note that the notion of expansivity also makes sense for general endomorphisms if $\varphi^kH$ is interpreted as a preimage for negative integers $k$. Nevertheless, we shall restrict this definition to the case of automorphisms, as stated above.

Let $t(G)$ denote the torsion part of $G$. Since $t(G)$ is the union of all finite subgroups of $G$, and finite sums of the form $\sum_k\varphi^kS$ are always contained in $t(G)$ when $S$ is a finite subgroup of $G$, expansivity and positive expansivity can be reformulated as follows.

\begin{prop}
 Let $\varphi\colon G\to G$ be an endomorphism, $\varphi|_{t(G)}\colon t(G)\to t(G)$ its restriction to $t(G)$, and $S$ a finite subgroup of $G$. The following statements hold.
 \begin{enumerate}
  \item $S$ is a positive generator for $\varphi$ if and only if $\tsum_{k\in\N}\varphi^kS=t(G)$.
  \item $\varphi$ is positively $S$-expansive if and only if $\varphi|_{t(G)}$ is positively $S$-expansive.
 \end{enumerate}
Furthermore, if $\varphi$ is an automorphism, the following also hold.
 \begin{enumerate}\setcounter{enumi}{2}
  \item $S$ is a generator for $\varphi$ if and only if $\tsum_{k\in\Z}\varphi^kS=t(G)$.
  \item $\varphi$ is $S$-expansive if and only if $\varphi|_{t(G)}$ is $S$-expansive.
 \end{enumerate}
\end{prop}

\begin{ex}\label{ex:bernoullishifts}
 The standard \emph{Bernoulli shifts} provide the basic examples of algebraic expansivity and positive expansivity. Let $I$ denote the set $\N$ or $\Z$. Given a finite group $S$, let $S^{(I)}=\bigoplus_{k\in I}S_k$ be the direct sum of copies $S_k=S$. Define the (forward) \emph{shift} map $\sigma\colon S^{(I)}\to S^{(I)}$ by $\sigma(s_k)=(s_{k-1})$ for $(s_k)\in S^{(I)}$, where for $I=\N$ we agree that $s_{-1}=0$. Consider also the subgroup $S_0\leq S^{(I)}$ defined as $S_0=\{(s_k)\in S^{(I)}:s_k=0\text{ if }k\neq0\}$ which clearly can be identified with $S$. Then it is easily checked that for $I=\N$, the (unilateral) shift is a positively expansive monomorphism  with positive generator $S_0$. Analogously, for $I=\Z$, the (bilateral) shift is an expansive automorphism with generator $S_0$.
\end{ex}

\begin{df}
 Let $\varphi\colon G \to G$ and $\psi\colon H \to H$ be endomorphisms. We say that $\psi$ is a \emph{factor} of $\varphi$ (or that $\varphi$ is \emph{semiconjugate} to $\psi$) if there exists an epimorphism $\pi\colon G \to H$ such that $\pi\circ\varphi=\psi\circ\pi$. The map $\pi$ is called a \emph{semiconjugacy}. If, in addition, $\pi$ is an isomorphism, we say that $\varphi$ and $\psi$ are conjugate, and $\pi$ is called a \emph{conjugation}.
\end{df}

\begin{prop}\label{prop:conjugacy}
 The following properties hold.
 \begin{enumerate}
  \item Any factor of a positively expansive endomorphism of a torsion group is positively expansive.
  \item Any factor automorphism of an expansive automorphism of a torsion group is expansive.
  \item Expansivity and positive expansivity are conjugacy invariants.
 \end{enumerate}
\end{prop}

\begin{proof}
 We will show only property (1) as the others follow similarly. Let $\varphi\colon G\to G$ be a positively expansive endomorphism and let $\pi\colon G\to H$ be a semiconjugacy to an endomorphism $\psi\colon H\to H$. Note that, since $G$ is torsion and $\pi$ is surjective, $H$ is torsion as well. Let $S\leq G$ be a positive generator for $\varphi$. Then
 \[
  \tsum_{k\in\N}\psi^k\pi S
  =\tsum_{k\in\N}\pi\varphi^kS
  =\pi\tsum_{k\in\N}\varphi^kS
  =\pi G
  =H,
 \]
 showing that $\psi$ is positively expansive with positive generator $\pi S$.
\end{proof}

\begin{prop}
  Let $G$ be a torsion group and $\varphi\colon G\to G$ an endomorphism. Then $\varphi$ is positively expansive if and only if it is a factor of a unilateral Bernoulli shift as defined in Example \ref{ex:bernoullishifts}. An analogous statement holds for the expansivity of automorphisms and bilateral Bernoulli shifts.
\end{prop}

\begin{proof}
 Given a positively $S$-expansive endomorphism (resp. an $S$-expansive isomorphism) $\varphi\colon G\to G$, let $\sigma\colon S^{(I)}\to S^{(I)}$ be the shift map as in Example \ref{ex:bernoullishifts}, where $I=\N$ (resp. $I=\Z$). Consider the map $q_\varphi\colon S^{(I)}\to G$ given by $q_\varphi(s_k)=\sum_{k\in I}\varphi^ks_k$ for $(s_k)\in S^{(I)}$. Then, since $\sum_{k\in I}\varphi^kS=G$ by expansivity, we have that $q_\varphi$ is an epimorphism. Furthermore, it is straightforward to verify that $q_\varphi\circ\sigma=\varphi\circ q_\varphi$, showing that $q_\varphi$ is a semiconjugacy from $\sigma$ to $\varphi$.
 
 \begin{center}
 \begin{tikzcd}[row sep=large, column sep=large]
  S^{(I)} \arrow[r, "\sigma"] \arrow[d, "q_\varphi"'] & S^{(I)} \arrow[d, "q_\varphi"] \\
  G \arrow[r, "\varphi"]                                     & G
 \end{tikzcd}
\end{center}
 
 Conversely, since the groups $S^{(I)}$ of Example 2.3 are torsion groups (because $S$ is finite), the factors of Bernoulli shifts inherit the corresponding expansivity property by Proposition \ref{prop:conjugacy}.
\end{proof}

\begin{prop}\label{prop:directsum}
 Let $\varphi\colon G \to G$ and $\psi\colon H \to H$ be endomorphisms, and denote by $\varphi\oplus\psi\colon G\oplus H\to G\oplus H$ their direct sum, that is, $(\varphi\oplus\psi)(g,h)=(\varphi g,\psi h)$ if $g\in G$, $h\in H$. Then $\varphi\oplus\psi$ is positively expansive iff $\varphi$ and $\psi$ are positively expansive. An analogous statement holds for expansive automorphisms.
\end{prop}

\begin{proof}
If $\varphi$ and $\psi$ are positively expansive, let $S_G$ and $S_H$ be positive generators for each, respectively. Then $S=S_G\oplus S_H$ is a finite subgroup of $G \oplus H$ which is easily seen to be a positive generator for $\varphi\oplus\psi$. Conversely, as any finite subgroup $S\leq G \oplus H$ is contained in $S_G \oplus S_H$ for some finite subgroups $S_G \leq G$ and $S_H \leq H$, when $S$ is a positive generator for $\varphi\oplus\psi$ we obtain positive generators $S_G$ and $S_H$ for $\varphi$ and $\psi$, respectively.
\end{proof}

\begin{prop}
 Let $G$ be a torsion group such that $G = H + K$ for subgroups $H,K\leq G$, and let $\varphi\colon G\to G$ be an endomorphism. If the restrictions $\varphi|_H$ and $\varphi|_K$ are positively expansive endomorphisms (resp. expansive automorphisms), then $\varphi$ is a positively expansive endomorphism (resp. expansive automorphism).
\end{prop}

\begin{proof}
 Consider the endomorphism $\varphi|_H \oplus \varphi|_K$ of the direct sum $H \oplus K$. Since both components are positively expansive (resp. expansive), it follows from Proposition \ref{prop:directsum} that $\varphi|_H\oplus\varphi|_K$ is positively expansive (resp. expansive).
 
 Now, consider the sum map $\pi\colon H \oplus K \to G$ defined by $\pi(h,k)=h+k$. Since $G=H+K$, $\pi$ is an epimorphism. Furthermore, as $\pi\circ(\varphi|_H\oplus \varphi|_K)=\varphi\circ\pi$, we deduce that $\pi$ is a semiconjugacy from $\varphi|_H \oplus \varphi|_K$ to $\varphi$. Since $G$ is torsion, it follows from Proposition \ref{prop:conjugacy} that $\varphi$ is positively expansive (resp. expansive).
\end{proof}

\begin{prop}\label{prop:extension}
 Let $\varphi\colon G\to G$ be an endomorphism of a torsion group $G$ and $H\leq G$ a subgroup such that the restriction $\varphi|_H\colon H\to H$ and the induced endomorphism $\widetilde{\varphi}\colon G/H\to G/H$ are positively expansive, then $\varphi$ is positively expansive. An analogous statement holds for expansive automorphisms.
\end{prop}

\begin{proof}
 Let $\pi\colon G \to G/H$ be the canonical projection. Since $\varphi|_H$ and $\widetilde{\varphi}$ are positively expansive, there exist positive generators $S_H \leq H$ and $S_{G/H}\leq G/H$, respectively. Since $G$ is torsion, there exists a finite subgroup $S_0 \leq G$ such that $\pi S_0=S_{G/H}$. We shall prove that $S=S_H + S_0$ is a positive generator for $\varphi$.
 
 Let $F \leq G$ be any finite subgroup. Since its image $\pi F$ is a finite subgroup of $G/H$ and $S_{G/H}$ is a positive generator for $\widetilde\varphi$, there exists $n_1\in\N$ such that
 \[
  \pi F
  \subseteq\tsum_{k=0}^{n_1}\widetilde{\varphi}^kS_{G/H}
  =\tsum_{k=0}^{n_1}\widetilde{\varphi}^k\pi S_0 
  =\pi\tsum_{k=0}^{n_1}\varphi^k S_0.
 \]
 This inclusion implies that for every $f\in F$, one can pick $x_f\in\sum_{k=0}^{n_1}\varphi^k S_0$ such that $f-x_f\in\ker\pi=H$. Let $E\leq H$ be the (finite) subgroup generated by all such differences, that is, $E=\langle f-x_f:f\in F\rangle$. Since $S_H$ is a positive generator for $\varphi|_H$, there exists $n_2\in\N$ such that $E\subseteq\sum_{k=0}^{n_2} \varphi^k S_H$.
 
 Finally, letting $n = \max\{n_1, n_2\}$, we obtain
 \[
  F 
  \subseteq\tsum_{k=0}^n\varphi^kS_0+E
  \subseteq\tsum_{k=0}^n\varphi^kS_0+\tsum_{k=0}^n\varphi^kS_H
  =\tsum_{k=0}^n\varphi^kS.
 \]
 Thus, $S$ is a positive generator for $\varphi$ as desired.
\end{proof}

\begin{prop}
 Let $\varphi\colon G\to G$ be an endomorphism and $n\in\N$, $n\neq 0$. Then $\varphi$ is positively expansive if and only if $\varphi^n$ is positively expansive. An analogous statement holds for expansive automorphisms.
\end{prop}

\begin{proof}
 We consider the case of automorphisms; the case of endomorphisms can be shown analogously. If $\varphi$ is expansive, then $\varphi^n$ is also expansive. Conversely, if $\varphi^n$ is expansive with generator $S$, then it is easy to verify that $\tsum_{|k|\leq n} \varphi^{k}(S)$ is a generator for $\varphi$.
\end{proof}

\section{Entropy and expansivity}

Next we recall the definition of algebraic entropy for group endomorphisms.

\begin{df}
 Let $\varphi\colon G\to G$ be a group endomorphism and $F\leq G$ a finite subgroup of $G$. We define the \emph{entropy of $\varphi$ relative to $F$} and the \emph{entropy of $\varphi$} as
 \[
  h_\textup{alg}(\varphi,F)=\textstyle\limsup_n\frac1n\log\bigl|\tsum_{k=0}^{n-1}\varphi^kF\bigr|
  \quad\text{and}\quad
  h_\textup{alg}(\varphi)=\sup_Fh_\textup{alg}(\varphi,F),
 \]
 respectively, where the $\sup_F$ is taken over all finite subgroups of $G$.
\end{df}

By \cite{Wei74}*{p.\,243} the sequence $n\mapsto\log\bigl|\tsum_{k=0}^{n-1}\varphi^kF\bigr|$ is subadditive, so that the $\limsup_n$ in the above definition can be replaced by a $\lim_n$, which is always finite.

\begin{lem}\label{lem:entropy}
 Let $\varphi\colon G\to G$ be an endomorphism and $E,F\leq G$ finite subgroups.
 \begin{enumerate}
  \item If $E\subseteq F$, then $h_\textup{alg}(\varphi,E)\leq h_\textup{alg}(\varphi,F)$.
  \item For any $m\in\N$, $h_\textup{alg}(\varphi,F)=h_\textup{alg}(\varphi,\sum_{k=0}^m\varphi^kF)$.
 \end{enumerate}
 If, in addition, $\varphi$ is an automorphism then
  \begin{enumerate}\setcounter{enumi}{2}
  \item $h_\textup{alg}(\varphi,F)=h_\textup{alg}(\varphi,\sum_{|k|\leq m}\varphi^kF)$ for any $m\in\N$.
 \end{enumerate}
\end{lem}

\begin{proof}
 Property (1) is easily established. For (2), let $E=\sum_{l=0}^m\varphi^lF$, and note that for $n\in\N$ we have $\tsum_{k=0}^{n-1}\varphi^kE=\tsum_{k=0}^{m+n-1}\varphi^kF$. Then
 \[
  h_\textup{alg}(\varphi,E)
  =\lim_n\tfrac1n\log\Bigl|\dsum_{k=0}^{n-1}\varphi^kE\Bigr|
  =\lim_n\tfrac{m+n}n\tfrac1{m+n}\log\Bigl|\dsum_{k=0}^{m+n-1}\varphi^kF\Bigr|
  =h_\textup{alg}(\varphi,F).
 \]
 Finally, (3) is obtained similarly by noting that $\sum_{|k|\leq m}\varphi^kF=\varphi^{-m}\sum_{k=0}^{2m}\varphi^kF$.
\end{proof}

\begin{thm}
 Let $\varphi\colon G \to G$ be a group endomorphism (resp. automorphism). If $S\leq G$ is a positive generator (resp. generator) for $\varphi$, then $h_\textup{alg}(\varphi)=h_\textup{alg}(\varphi, S)$.
 
 In particular, any positively expansive endomorphism or expansive automorphism has finite algebraic entropy.
\end{thm}

\begin{proof}
 If $S$ is a positive generator for $\varphi$, then for every finite subgroup $F\leq G$ there exists $m\in\N$ such that $F\subseteq\sum_{k=0}^m\varphi^kS$. Therefore, by Lemma \ref{lem:entropy}, 
 \[
  h_\textup{alg}(\varphi,F)\leq h_\textup{alg}(\varphi,\tsum_{k=0}^m\varphi^kS)=h_\textup{alg}(\varphi,S).  
 \]
Whence, $h_\textup{alg}(\varphi)=\sup_Fh_\textup{alg}(\varphi,F)=h_\textup{alg}(\varphi,S)$, which is finite. The case for expansive automorphisms follows similarly using part (3) of Lemma \ref{lem:entropy}.
\end{proof}

\section{Positively expansive epimorphisms}

\begin{lem}\label{lem:autogarche}
 If $\varphi$ is a positively expansive epimorphism of a torsion group $G$, then there exists a finite subgroup $K\leq G$ with the following property: for every finite subgroup $F\leq G$ there exists $m\in\N$ such that $F\subseteq\varphi^mK$.
\end{lem}

\begin{proof}
 Let $S\leq G$ be a positive generator for $\varphi$. Since $S$ is finite, $\varphi$ is surjective, and $G$ is torsion, there exists a finite subgroup $S_0\leq G$ such that $\varphi S_0=S$. By positive expansivity, there exists $n\in\N$ such that $S_0\subseteq\sum_{k=0}^n\varphi^kS$.  
 
 We will give an inductive proof of the following claim:
 \begin{equation}\label{eq:autogarche}
   \tsum_{k=0}^{n+m}\varphi^kS=\tsum_{k=m}^{n+m}\varphi^kS \quad\text{for all}\quad m\in\N.
 \end{equation}
 Indeed, if $m=0$ the equality is trivial. If the equality is true for an $m\in\N$, then
 \[
  \tsum_{k=m+1}^{n+m+1}\varphi^kS
  =\varphi\tsum_{k=m}^{n+m}\varphi^kS
  =\varphi\tsum_{k=0}^{n+m}\varphi^kS
  =\tsum_{k=1}^{n+m+1}\varphi^kS.
 \]
 Now, since $S_0\subseteq\sum_{k=0}^n\varphi^kS$, we also have
 \[
  S=\varphi S_0
  \subseteq\varphi\tsum_{k=0}^{n}\varphi^kS
  =\tsum_{k=1}^{n+1}\varphi^kS
  \subseteq\tsum_{k=1}^{n+m+1}\varphi^kS,
 \]
 from which we deduce that $\tsum_{k=0}^{n+m+1}\varphi^kS=\tsum_{k=1}^{n+m+1}\varphi^kS$, which proves the claim.
 
 Finally, taking $K=\sum_{k=0}^n\varphi^kS$, claim (\ref{eq:autogarche}) translates into $\varphi^mK=\tsum_{k=0}^{n+m}\varphi^kS$ for all $m\in\N$, and the result follows from the fact that $S$ is a positive generator.
\end{proof}

\begin{thm}\label{thm:positexpepiis finite}
 If a torsion group $G$ supports a positively expansive epimorphism then $G$ is finite.
\end{thm}

\begin{proof}
 By Lemma \ref{lem:autogarche}, there exists a finite subgroup $K \leq G$ such that every finite subgroup $F \leq G$ is contained in $\varphi^m K$ for some $m \in \mathbb{N}$. Then $|F|$ is uniformly bounded by $|K|$. Thus $G$ is finite and $|G| = |K|$.
\end{proof}

\begin{ex}
 We will show an example of a positively expansive epimorphism supported on an infinite group.
 Define $G=\toplus_{n\in\Z}G_n$ where $G_n=\Z$ if $n<0$ and $G_n=\Z_2$ if $n\geq0$. Let $\varphi\colon G\to G$ be the shift-like map given by 
 \[
 (\ldots,g_{-2},g_{-1},\stackrel{\vee}{g_0},g_1,g_2,\ldots)
 \stackrel{\varphi}{\longmapsto}
 (\ldots,g_{-3},g_{-2},\stackrel{\vee}{\pi g_{-1}},g_0,g_1,\ldots),
 \]
 where the ``$\stackrel{\vee}{}$'' marks the 0-th coordinate of each bi-sequence and $\pi$ is the canonical epimorphism $\Z\to\Z_2$. It is easy to check that the subgroup consisting of the bi-sequences supported on the 0-th coordinate is a positive expansivity subgroup for $\varphi$. Of course, in view of Theorem \ref{thm:positexpepiis finite}, this is possible since $G$ is not torsion. Note also that the positively expansive endomorphism of the torsion part of $G$, namely $t(G)=\toplus_{n\in\N}G_n$, induced by $\varphi$ is not surjective in this case.
\end{ex}

\section{Connection to topological expansivity in the dual group}

For continuous endomorphisms $\psi$ of a compact group $K$, topological expansivity is expressed in terms of neighborhoods of the identity $e \in K$ \cites{Eis66,Wu_66} as follows: $\psi$ is said to be positively expansive if there exists a neighborhood $U$ of $e$ such that $\bigcap_{k \in \mathbb{N}} \psi^{-k}(U) = \{e\}$. Analogously, if $\psi$ is an automorphism, then $\psi$ is expansive if there exists a neighborhood $U$ of $e$ such that $\bigcap_{k \in \mathbb{Z}} \psi^{-k}(U) = \{e\}$.

Let $G$ be a group and $\widehat{G}$ its \emph{Pontryagin dual}. Given a subgroup $H \leq G$ consider its \emph{annihilator} \quad $H^\perp = \{ \chi \in \widehat{G} : \chi|_H=0\}\leq\widehat{G}$,\quad and for an endomorphism $\varphi\colon G\to G$ consider its \emph{dual endomorphism} \quad $\widehat\varphi\colon\widehat{G}\to\widehat{G}$,\quad $\widehat\varphi(\chi)=\chi\circ\varphi$, for $\chi\in\widehat G$. Then it is known that \quad $(\varphi H)^\perp = \widehat{\varphi}^{-1}H^\perp$,\quad where for $K \leq \widehat{G}$, $K^\perp = \bigcap_{\chi \in K} \ker \chi$ is its \emph{annihilator in} $G$. Moreover, for a family of subgroups $\{H_i\}$ of $G$, resp. $\{K_i\}$ of $\widehat G$, we have the formulas \quad $(\sum_i H_i)^\perp = \bigcap_i H_i^\perp$,\quad resp. \quad $(\bigcap_i K_i)^\perp = \overline{\sum_i K_i^\perp}$.

\begin{thm}
 The algebraic expansivity on (discrete) torsion groups is the Pontryagin dual concept of topological expansivity on $0$-dimensional compact groups.
\end{thm}

\begin{proof}
 We will consider only the case of positive expansivity for endomorphisms, as the case of expansivity for automorphisms is analogous. By the duality theory, we know that torsion groups correspond to totally disconnected compact groups. Moreover, for a torsion group $G$, the family of annihilators $H^\perp$ of finite subgroups $H \leq G$ forms a basis of clopen neighborhoods of the identity $e \in \widehat{G}$.

 Let $\varphi\colon G\to G$ be an endomorphism of a torsion group $G$. 

 Suppose first that $\varphi$ is (algebraically) positively expansive. Then there exists a finite subgroup $S\leq G$ such that $G=\sum_{n\in\N} \varphi^n(S)$. Since
 \[
  \{e\}
  =G^\perp
  =(\tsum_{n\in\N}\varphi^nS)^\perp
  =\tcap_{n\in\N}(\varphi^nS)^\perp
  =\tcap_{n\in\N}\widehat\varphi^{-n}S^\perp,
 \]
 we conclude that $U=S^\perp$ is a neighborhood of the identity in $\widehat{G}$ satisfying the definition of positive topological expansivity for $\widehat\varphi$.

 Conversely, suppose that $\widehat{\varphi}$ is positively topologically expansive. Then there exists an open neighborhood $U$ of $e\in\widehat{G}$ such that $\bigcap_{n\in\N}\widehat{\varphi}^{-n}U = \{e\}$. Since $\widehat G$ is zero-dimensional, we may assume without loss of generality that $V$ is a clopen subgroup. Let $S=U^\perp$ be its \emph{annihilator in $G$}. Since $V$ is an open subgroup of a compact group, $S$ is a finite subgroup of $G$. By applying the dual formula for the intersection, we have
 \[
  G
  =\{e\}^\perp
  =\bigl(\tcap_{n\in\N}\widehat\varphi^{-n}U\bigr)^\perp
  =\overline{\tsum_{n\in\N}(\widehat{\varphi}^{-n}U)^\perp}
  =\tsum_{n\in\N}\varphi^nU^\perp,
 \] 
 from which we obtain that $S=U^\perp$ is a positive generator for $\varphi$
\end{proof}

The duality established in the previous result leads to a noteworthy consequence regarding the stability of expansivity. In \cite{Wil11}*{Proposition 6.1}, it is proven that the quotient of a topologically expansive automorphism on a 0-dimensional compact group remains expansive. By the Theorem above, this result translates into the fact that the restriction of an algebraically expansive automorphism to an invariant subgroup is also algebraically expansive. While the original proof in \cite{Wil11} is far from trivial, the duality established here suggests the possibility of a more direct, purely algebraic approach. Developing such a proof remains an interesting problem to further clarify the connection between these two notions of expansivity.

\begin{thm}
 Let $\varphi \colon G \to G$ be an expansive automorphism of a torsion group $G$. If $H \leq G$ is a subgroup such that $\varphi H=H$, then the restriction $\varphi|_H \colon H \to H$ is also algebraically expansive.
\end{thm}

\begin{conj}
 Let $\varphi \colon G \to G$ be an expansive endomorphism of a torsion group $G$. If $H \leq G$ is a subgroup such that $\varphi H\subseteq H$, then the restriction $\varphi|_H \colon H \to H$ is also positively expansive.
\end{conj}

\vspace{10mm}

\noindent Mauricio Achigar Pereira\\
{\tt machigar@litoralnorte.udelar.edu.uy}\\
{\sc Departamento de Matemática y Estadística del Litoral}\\
{\sc Centro Universitario Regional Litoral Norte}\\
{\sc Universidad de la República}\\
25 de Agosto 281, Salto (50000), Uruguay


\begin{bibdiv}
\begin{biblist}


\bib{Ach19}{thesis}{
author={M. Achigar},
title={\href{https://hdl.handle.net/20.500.12008/23282}{Din{\'a}mica topol{\'o}gica expansiva: algunos aportes}},
type={PhD Thesis},
organization={Universidad de la Rep{\'u}blica (Uruguay)},
address={Facultad de Ciencias - PEDECIBA},
year={2019}}


\bib{Ach21}{article}{
author={M. Achigar},
title={\href{http://doi.org/10.1016/j.topol.2020.107577}{Expansive systems on lattices}},
journal={Topol. Appl.},
volume={290},
number=={},
year={2021},
pages={107577}}


\bib{Ach23}{article}{
author={M. Achigar},
title={\href{https://arxiv.org/abs/2203.10394}{Abstract entropy and expansiveness}},
journal={arXiv:2203.10394v2},
year={2023},
note={Preprint}}


\bib{AAM16}{article}{
author={M. Achigar},
author={A. Artigue},
author={I. Monteverde},
title={\href{https://doi.org/10.1016/j.topol.2016.04.016}{Expansive homeomorphisms on non-Hausdorff spaces}},
journal={Topol. Appl.},
volume={207},
year={2016},
pages={109--122}}


\bib{AKM65}{article}{
author={L. Adler},
author={G. Konheim},
author={M. McAndrew},
title={\href{https://doi.org/10.2307/1994177}{Topological entropy}},
journal={Trans. Amer. Math. Soc.},
volume={114},
year={1965},
pages={309--319}}




\bib{Bry60}{article}{
author={B. F. Bryant},
title={\href{https://doi.org/10.2140/PJM.1960.10.1163}{On expansive homeomorphisms}},
journal={Pac. J. Math.},
volume={10},
year={1960},
pages={1163--1167}}


\bib{Bry62}{article}{
author={B. F. Bryant},
title={\href{https://doi.org/10.2307/2312129}{Expansive self-homeomorphisms of a compact metric space}},
journal={Am. Math. Mon.},
volume={69},
number={5},
year={1962},
pages={386--391}}


\bib{CoK06}{article}{
author={E. M. Coven},
author={M. Keane},
title={\href{https://doi.org/10.1214/074921706000000310}{Every compact metric space that supports a positively expansive homeomorphism is finite}},
journal={Dynamics \& Stochastics, IMS Lecture Notes Monogr. Ser.},
volume={48},
pages={304--305},
year={2006}}


\bib{DGB13}{article}{
author={D. Dikranjan},
author={A. Giordano Bruno},
title={\href{https://doi.org/10.1285/i15900932v33n1p1}{Discrete dynamical systems in group theory}},
journal={Note Mat.},
volume={33},
number={1},
pages={1--48},
year={2013}}


\bib{DGB16}{article}{
author={D. Dikranjan},
author={A. Giordano Bruno},
title={\href{https://doi.org/10.1016/j.aim.2016.04.020}{Entropy on abelian groups}},
journal={Adv. Math.},
volume={298},
pages={612--653},
year={2016}}


\bib{DGB19}{article}{
author={D. Dikranjan},
author={A. Giordano Bruno},
title={\href{https://doi.org/10.4064/dm791-2-2019}{Entropy on normed semigroups (towards a unifying approach to entropy)}},
journal={Diss. Math.},
volume={542},
pages={1--90},
year={2019}}






\bib{Eis66}{article}{
title={\href{https://doi.org/10.4064/FM-59-3-313-321}{Expansive transformation semigroups of endomorphisms}},
author={Murray Eisenberg},
journal={Fund. Math.},
year={1966},
volume={59},
pages={313-321}}


\bib{Fri87}{article}{
author={D. Fried},
title={\href{https://doi.org/10.1017/S014338570000417X}{Finitely presented dynamical systems}},
journal={Ergod. Theory Dyn. Syst.},
year={1987},
volume={7},
number={4},
pages={489--507}}




\bib{GoH55}{book}{
author={W. H. Gottschalk},
author={G. A. Hedlund},
title={\href{https://doi.org/10.1090/coll/036}{Topological dynamics}},
volume={36},
series={Colloquium Publications},
publisher={Amer. Math. Soc.},
year={1955}}




\bib{KeR69}{article}{
author={H. Keynes},
author={J. Robertson},
title={\href{https://doi.org/10.1007/BF01695625}{Generators for topological entropy and expansivity}},
journal={Math. Syst. Theory},
volume={3},
year={1969},
pages={51--59}}


\bib{LiW00}{article}{
author={E. Lindenstrauss},
author={B. Weiss},
title={\href{https://doi.org/10.1007/BF02810577}{Mean topological dimension}},
journal={Isr. J. Math.},
volume={115},
pages={1--24},
year={2000}}




\bib{Utz50}{article}{
author={Utz, W. R.},
title={\href{https://doi.org/10.2307/2031982}{Unstable homeomorphisms}},
journal={Proc. Amer. Math. Soc.},
volume={1},
number={6},
pages={769--774},
year={1950}}




\bib{Wei74}{article}{
author={M. D. Weiss},
title={\href{https://doi.org/10.1007/BF01762672}{Algebraic and other entropies of group endomorphisms}},
journal={Math. Syst. Theory},
volume={8},
number={3},
pages={243--248},
year={1974}}


\bib{Wil11}{article}{
  title={\href{https://doi.org/10.1017/etds.2012.172}{The nub of an automorphism of a totally disconnected, locally compact group}},
  author={G. A. Willis},
  journal={Ergod. Theory Dyn. Syst.},
  year={2011},
  volume={34},
  pages={1365--1394},
}


\bib{Wu_66}{article}{
author={T. S. Wu},
title={Expansive automorphisms in compact groups}, volume={18},
journal={Math. Scand.},
pages={23–24},
year={1966}}


\end{biblist}
\end{bibdiv}
\end{document}